\documentclass[10pt]{amsart}

\usepackage{amssymb,amsmath,amsthm,amsfonts}

\theoremstyle{plain}
\newtheorem{Theorem}{Theorem}
\newtheorem{Proposition}[Theorem]{Proposition}

\theoremstyle{definition}

\begin{document}
\title[Error estimates for alternating series]{Error estimates for the Gregory-Leibniz series and the alternating harmonic series using Dalzell integrals}
\author{Diego Rattaggi}
\email{rattaggi@gmx.ch}
\date{\today}
\begin{abstract}
The computation of Dalzell integrals $\int_0^1 \frac{x^m (1-x)^n}{1+x^2} \, dx > 0$
gives new error estimates for the partial sums of the Gregory-Leibniz series $1 - \frac{1}{3} + \frac{1}{5} - \frac{1}{7} \pm \ldots$ and for the alternating harmonic series
$1 - \frac{1}{2} + \frac{1}{3} - \frac{1}{4} \pm \ldots$
\end{abstract}
\maketitle
\section{Introduction}
Dalzell (\cite{Dalz}) observed that
\[
0 < \int_0^1 \frac{x^4 (1-x)^4}{1+x^2} \, dx =  \frac{22}{7} - \pi.
\]
Backhouse (\cite{Back}) generalized Dalzell's integral to the infinite family
\[
I_{m,n} := \int_0^1 \frac{x^m (1-x)^n}{1+x^2} \, dx \quad \quad  (m,n \in \mathbb{N})
\]
to get better rational approximations of $\pi$, e.g.\
\[
0 < \frac{ I_{32,32}}{16384} 
= \pi - \frac{19809071774292917047896724979}{6305423381881718760060595200} \approx 4 \cdot 10^{-25},\\[5mm]
\]
see also Lucas (\cite{Lucas}).
Moreover, Backhouse showed that the integral $I_{m,n}$ always leads to a rational approximation of~$\pi$, if $2m-n \equiv 0 \pmod 4$. 
Under this condition, we observed by computing several integrals $I_{m,n}$ by hand,
that by fixing an even $n$, we not only get approximations of $\pi$, but also good error estimates for the partial sums of the Gregory-Leibniz series
\[
1 - \frac{1}{3} + \frac{1}{5} - \frac{1}{7} \pm \ldots = \frac{\pi}{4}.
\]
Using $I_{m,n} > 0$, elementary computations immediately lead to an upper and a lower bound for that error.
To illustrate this, we start with the simplest case $n=2$ and $m$ odd.

\section{Estimates for the Gregory-Leibniz series} \label{S2}
If we denote by $GLS_k$ the $k$th partial sum of the Gregory-Leibniz series, i.e.\ 
\[
GLS_k  := \sum_{i=1}^{k} (-1)^{i+1} \frac{1}{2i-1},
\]
then we obtain a first estimate for the error $\left|   \frac{\pi}{4} - GLS_k  \right|$. 

\begin{Proposition} \label{P1}
\[
 \frac{2k+3}{8k^2+12k+4} < \left|   \frac{\pi}{4} - GLS_k  \right| < \frac{1}{4k}
\]
\end{Proposition}

\begin{proof}
To prove the upper bound, let $n = 2$ and $m \equiv 3 \pmod 4$. Then
\begin{align}
\frac{ I_{m,2}}{2}  &= \frac{1}{2} \int_0^1 \frac{x^m (1-x)^2}{1+x^2} \, dx = \frac{1}{2} \int_0^1 \frac{x^{m+2}- 2x^{m+1}+x^m}{1+x^2} \, dx \notag \\
&= \frac{1}{2} \int_0^1  x^m -2x^{m-1} +2x^{m-3} - 2x^{m-5} \pm \ldots  + 2 - \frac{2}{1+x^2} \, dx \notag \\
&= \left[ \frac{x^{m+1}}{2(m+1)} - \frac{x^m}{m} + \frac{x^{m-2}}{m-2} -\frac{x^{m-4}}{m-4} \pm \ldots + x - \arctan(x)\right]_0^1 \notag \\
&= \frac{1}{2(m+1)} - \frac{1}{m} + \frac{1}{m-2} - \frac{1}{m-4} \pm \ldots + 1 - \frac{\pi}{4} \notag 
\end{align}
Since obviously $I_{m,n} > 0$, we get
\[
\frac{\pi}{4} - \left( 1- \frac{1}{3} + \frac{1}{5} - \frac{1}{7} \pm \ldots - \frac{1}{m} \right) < \frac{1}{2(m+1)}
\]
and with $m = 2k-1$
\[
\frac{\pi}{4} - GLS_k < \frac{1}{4k}\\[5mm]
\]
In the other case $n=2$ and $m \equiv 1 \pmod 4$, we get in the same way
\[
\left( 1- \frac{1}{3} + \frac{1}{5} - \frac{1}{7} \pm \ldots + \frac{1}{m} \right) - \frac{\pi}{4} < \frac{1}{2(m+1)}
\]
and
\[
GLS_k - \frac{\pi}{4} <  \frac{1}{4k}
\]
The computation for the upper bound also immediately leads to a lower bound by separating the last summand $\frac{1}{m}$ from $1 -  \frac{1}{3} + \ldots $.
More precisely, we have 
in the case $m \equiv 3 \pmod 4$ as seen before
\[
\frac{1}{2(m+1)} - \frac{1}{m} + \frac{1}{m-2} - \frac{1}{m-4} \pm \ldots + 1 - \frac{\pi}{4} > 0
\]
hence
\[
\frac{1}{m-2} - \frac{1}{m-4} \pm \ldots + 1 - \frac{\pi}{4} > \frac{1}{m} - \frac{1}{2(m+1)}
\]
and
\[
\left( 1- \frac{1}{3} \pm \ldots - \frac{1}{m-4} + \frac{1}{m-2} \right)  - \frac{\pi}{4} > \frac{1}{m} - \frac{1}{2(m+1)} 
\]
The sum in brackets has $k= \frac{m-1}{2}$ summands, therefore we replace $m$ by $2k+1$ and get
\[
GLS_k - \frac{\pi}{4} > \frac{1}{2k+1}-\frac{1}{2(2k+2)} = \frac{2k+3}{8k^2+12k+4}
\]
In the second case $m \equiv 1 \pmod 4$, we similarly get
\[
\frac{1}{m} - \frac{1}{2(m+1)} < \frac{\pi}{4} - \left( 1- \frac{1}{3}  \pm \ldots + \frac{1}{m-4}  - \frac{1}{m-2} \right)
\]
and the claim follows.
\end{proof}

Proposition~\ref{P1} can be improved increasing $n$ to $4$.
\begin{Proposition} \label{P2}
\[
\frac{1}{4} \left( -\frac{1}{2k+5}+\frac{4}{2k+4}- \frac{5}{2k+3} + \frac{4}{2k+1} \right) < \left|   \frac{\pi}{4} - GLS_k  \right| < \frac{1}{4} \left( \frac{1}{2k+3} - \frac{4}{2k+2} + \frac{5}{2k+1}  \right)
\]
or equivalently
\[
\frac{4k^3+26k^2+58k+47}{16k^4+104k^3+236k^2+214k+60} < \left|   \frac{\pi}{4} - GLS_k  \right| < \frac{2k^2+6k+5}{8k^3+24k^2+22k+6}
\]
\end{Proposition}

\begin{proof}
Let $n=4$ and let $m$ be even (such that $2m-n \equiv 0 \pmod 4$).
\begin{align}
\frac{ I_{m,4}}{4}  &=  \int_0^1 \frac{x^{m} (1-x)^4}{4(1+x^2)} \, dx 
= \int_0^1 \frac{x^{m}}{4} \cdot \frac{x^4 - 4x^3 + 6x^2 - 4x + 1}{1+x^2}  \, dx \notag \\
&=  \int_0^1 \frac{x^{m}}{4} \cdot \left(x^2-4x+5 - \frac{4}{1+x^2}\right)  \, dx 
= \int_0^1 \frac{1}{4} \left( x^{m+2}-4x^{m+1}+5x^m\right) - \frac{x^{m}}{1+x^2}   \, dx \notag 
\end{align}

If $m \equiv 2 \pmod 4$, the computation continues like
\begin{align}
&=  \int_0^1 \frac{1}{4} \left(x^{m+2}-4x^{m+1}+5x^m \right) - x^{m-2} + x^{m-4} - x^{m-6}\pm \ldots - 1 + \frac{1}{1+x^2}   \, dx \notag \\
&= \left[ \frac{1}{4} \left( \frac{x^{m+3}}{m+3}-\frac{4x^{m+2}}{m+2}+ \frac{5x^{m+1}}{m+1}\right) - \frac{x^{m-1}}{m-1} + \frac{x^{m-3}}{m-3} - \frac{x^{m-5}}{m-5} \pm \ldots - x + \arctan(x) \right]_0^1 \notag \\
&= \frac{1}{4} \left( \frac{1}{m+3}-\frac{4}{m+2}+ \frac{5}{m+1}\right) - \frac{1}{m-1} + \frac{1}{m-3} - \frac{1}{m-5}  \pm \ldots - 1 + \frac{\pi}{4} > 0 \notag
\end{align}

Therefore
\[
\left( 1 - \frac{1}{3} \pm \ldots + \frac{1}{m-1} \right) - \frac{\pi}{4} <  \frac{1}{4} \left( \frac{1}{m+3}-\frac{4}{m+2}+ \frac{5}{m+1} \right)
\]
In the other case $m \equiv 0 \pmod 4$, we similarly get
\[
\frac{\pi}{4} - \left( 1 - \frac{1}{3} \pm \ldots - \frac{1}{m-1} \right) 
< \frac{1}{4} \left( \frac{1}{m+3}-\frac{4}{m+2}+ \frac{5}{m+1} \right)
\]
The substitution $m=2k$ completes the proof for the upper bound.

To get the lower bound, we write in the case $m \equiv 2 \pmod 4$
\[
\frac{1}{m-3} - \frac{1}{m-5}  \pm \ldots - 1 + \frac{\pi}{4} > 
\frac{1}{m-1} - \frac{1}{4} \left( \frac{1}{m+3}-\frac{4}{m+2}+ \frac{5}{m+1}\right)
\]
hence
\[
\frac{1}{4} \left( -\frac{1}{m+3}+\frac{4}{m+2}- \frac{5}{m+1} + \frac{4}{m-1} \right) < \frac{\pi}{4} - \left( 1 - \frac{1}{3} \pm \ldots - \frac{1}{m-3} \right) 
\]
The substitution $m=2k+2$ gives
\[
\frac{1}{4} \left( -\frac{1}{2k+5}+\frac{4}{2k+4}- \frac{5}{2k+3} + \frac{4}{2k+1} \right) < \frac{\pi}{4} - GLS_k 
\]
In the remaining case $m \equiv 0 \pmod 4$, we obtain in the same way
\[
\frac{1}{4} \left( -\frac{1}{2k+5}+\frac{4}{2k+4}- \frac{5}{2k+3} + \frac{4}{2k+1} \right) < GLS_k - \frac{\pi}{4}
\]
\end{proof}

These results can be further improved by taking 
$n=6$, $m$ odd (Proposition \ref{P3}),
$n=8$, $m$ even (Proposition \ref{P4}), and so on.
Their proofs would use exactly the same ideas as the proof of Proposition \ref{P2}.

\begin{Proposition} \label{P3}
\begin{align}
\left|   \frac{\pi}{4} - GLS_k  \right| &<  
\frac{1}{8} \left(
\frac{1}{2k+6} - \frac{6}{2k+5} + \frac{14}{2k+4} - \frac{14}{2k+3} + \frac{1}{2k+2} + \frac{8}{2k+1}   \right)
\notag \\
&= \frac{16k^5+168k^4+696k^3+1428k^2+1454k+567}{64k^6+672k^5+2800k^4+5880k^3+6496k^2+3528k+720}\notag
\end{align}

and

\begin{align}
\left|   \frac{\pi}{4} - GLS_k  \right| &> \frac{1}{8} \left( - \frac{1}{2k+4} + \frac{6}{2k+3} - \frac{14}{2k+2} + \frac{14}{2k+1} - \frac{1}{2k}   \right) \notag \\
&= \frac{8k^4+40k^3+68k^2+40k-3}{32k^5+160k^4+280k^3+200k^2+48k} \notag
\end{align}
\end{Proposition}

\begin{Proposition} \label{P4}
\[
\left|   \frac{\pi}{4} - GLS_k  \right| < 
\frac{1}{16} \left( \frac{1}{2k+9} - \frac{8}{2k+8} + \frac{27}{2k+7} - \frac{48}{2k+6} + \frac{43}{2k+5} - \frac{8}{2k+4} - \frac{15}{2k+3} + \frac{16}{2k+1} \right) 
\]
\[
= \frac{16k^7+344k^6+3132k^5+15678k^4+46730k^3+83320k^2+82854k+35631}{64k^8+1376k^7+12544k^6+63056k^5+190036k^4+348614k^3+375066k^2+211284k+45360} \\[2mm]
\]
and
\[
\left|   \frac{\pi}{4} - GLS_k  \right| > \frac{1}{16} \left( -\frac{1}{2k+7} + \frac{8}{2k+6} - \frac{27}{2k+5} + \frac{48}{2k+4} - \frac{43}{2k+3} + \frac{8}{2k+2} + \frac{15}{2k+1}  \right) 
\]
\[
= \frac{8k^6+112k^5+642k^4+1932k^3+3226k^2+2828k+981}{32k^7+448k^6+2576k^5+7840k^4+13538k^3+13132k^2+6534k+1260}
\]
\end{Proposition}

\section{Comparison with other estimates}
We compare some error estimates for general alternating series with our estimates.
The orignal error estimate coming from the Leibniz criterion for alternating series leads to
\[
\left|   \frac{\pi}{4} - GLS_k  \right| \leq \frac{1}{2k+1}
\]
This was improved by Calabrese (\cite{Cala}).
For the Gregory-Leibniz series, it gives
\[
\frac{1}{4k+2} < \left|   \frac{\pi}{4} - GLS_k  \right| < \frac{1}{4k-2}
\]
This result was again refined by Johnsonbaugh (\cite{John}):
Let $a_1-a_2+a_3-a_4 \pm \ldots$ be an alternating series.
Define 
\[
\Delta^1 a_k := a_k - a_{k+1} \quad \text{ and } \quad  \Delta^r a_k := \Delta^{r-1} a_k -  \Delta^{r-1} a_{k+1}
\]
for $r>1$. If all the sequences $(\Delta^r a_k)$ for $r=1, 2, 3, \ldots, j$ decrease monotonically to zero,
then Johnsonbaugh showed for the error $R_k$, that
\[
\frac{a_{k+1}}{2} + \frac{\Delta^1 a_{k+1}}{2^2} + \ldots +  \frac{\Delta^j a_{k+1}}{2^{j+1}} < |R_k| < \frac{a_k}{2} - \left(   \frac{\Delta^1 a_k}{2^2} + \ldots +  \frac{\Delta^j a_k}{2^{j+1}} \right),
\] 
see \cite[Theorem 3]{Villa}.
For the Gregory-Leibniz series, this gives for example
\[
a_k = \frac{1}{2k-1}
\]
\[
\Delta^1 a_k = a_k - a_{k+1} =  \frac{1}{2k-1} -  \frac{1}{2k+1} = \frac{2}{4k^2-1}
\]
and
\[
\Delta^2 a_k = \Delta^1 a_k -  \Delta^1 a_{k+1} = \frac{1}{2k-1} -  \frac{1}{2k+1} - \frac{1}{2k+1} +  \frac{1}{2k+3} = \frac{8}{8k^3+12k^2-2k-3}
\]
So, we obtain for $j=1$
\[
\frac{1}{2(2k+1)} + \frac{1}{4(2k+1)} - \frac{1}{4(2k+3)} <  \left|   \frac{\pi}{4} - GLS_k  \right| < \frac{1}{2(2k-1)} - \frac{1}{4(2k-1)} + \frac{1}{4(2k+1)}  
\]
hence
\[
\frac{k+2}{4k^2+8k+3} <  \left|   \frac{\pi}{4} - GLS_k  \right| < \frac{k}{4k^2-1}
\]
It is easy to check, that these bounds are worse than the bounds of Proposition \ref{P1}.

Similarly, we get for $j=2$ 
\[
\frac{2k^2+9k+11}{8k^3+36k^2+46k+15}<  \left|   \frac{\pi}{4} - GLS_k  \right| < \frac{2k^2+3k-1}{8k^3+12k^2-2k-3}
\]
These bounds are worse than the bounds of Proposition \ref{P2}. For example comparing the two upper bounds we have
\[
 \frac{2k^2+6k+5}{8k^3+24k^2+22k+6} < \frac{2k^2+3k-1}{8k^3+12k^2-2k-3}
\]
since
\[
(2k^2+3k-1)(8k^3+24k^2+22k+6) - (2k^2+6k+5)(8k^3+12k^2-2k-3) = 12k^2 + 24k +9
\]
is always positive.

The following two tables show some numerical comparisons for the different error estimates 
(our propositions and Johnsonbaughs error estimates up to $j=5$), taking $k=10$ and $k=20$.

\begin{table}[ht]
\begin{tabular}{l|c|c}
 & $k=10$ & $k=20$ \\ 
Leibniz & 0.047619047619 & 0.024390243902\\
Calabrese & 0.026315789474 & 0.012820512821\\
Johnsonbaugh ($j=1$) & 0.025062656642 & 0.012507817386\\
Proposition \ref{P1} & 0.025000000000 & 0.012500000000 \\
Johnsonbaugh ($j=2$) & 0.024953688569 & 0.012493273412 \\
Johnsonbaugh ($j=3$) & 0.024940612401 & 0.012492303814 \\
Proposition \ref{P2} & 0.024938829287 & 0.012492234557 \\
Johnsonbaugh ($j=4$) & 0.024938675190 & 0.012492221295\\
Johnsonbaugh ($j=5$) & 0.024938341189 & 0.012492212875\\
Proposition \ref{P3} & 0.024938268253 & 0.012492211870\\
Proposition \ref{P4} & 0.024938258893 & 0.012492211732\\
True error & $0.024938258665$ & 0.012492211731 \\
\end{tabular}
\caption{Upper bounds for $k=10$ and $k=20$} 
\label{Table1}
\end{table}

\begin{table}[ht]
\begin{tabular}{l|c|c}
 & $k=10$ & $k=20$ \\ 
Calabrese & 0.023809523810 & 0.012195121951 \\
Johnsonbaugh ($j=1$) & 0.024844720497 & 0.012478729438 \\
Proposition \ref{P1} & 0.024891774892 & 0.012485481998 \\
Johnsonbaugh ($j=2$) & 0.024927536232 & 0.012491334216 \\
Johnsonbaugh ($j=3$) & 0.024936737980 & 0.012492138776 \\
Proposition \ref{P2} & 0.024937888199 & 0.012492193632 \\
Johnsonbaugh ($j=4$) & 0.024938007187 & 0.012492204454 \\
Johnsonbaugh ($j=5$) & 0.024938211898 & 0.012492210893 \\
Proposition \ref{P3} & 0.024938241107 & 0.012492211537 \\
Proposition \ref{P4} & 0.024938258199 & 0.012492211728 \\
True error & $0.024938258665$ & 0.012492211731 \\
\end{tabular}
\caption{Lower bounds for $k=10$ and $k=20$} 
\label{Table2}
\end{table}

\section{Related series}
As observed by Backhouse (\cite{Back}), the integral $I_{m,n}$ leads to a rational approximation of $\ln(2)$, if $2m-n \equiv 2 \pmod 4$.
In these cases, we now directly get error estimates for the series
\[
\ln(\sqrt{2}) = \frac{1}{2} - \frac{1}{4} + \frac{1}{6} - \frac{1}{8} \pm \ldots 
\]
Indeed, all the computations done in Section \ref{S2} work analogously here, replacing 
\[
\int_{0}^{1} \frac{1}{1+x^2} \, dx \quad \text{ by } \quad \int_{0}^{1} \frac{x}{1+x^2} \, dx
\]
hence replacing $\arctan(x)$ by $\frac{1}{2} \ln(1+x^2)$ and therefore replacing $\frac{\pi}{4}$ by $\frac{1}{2} \ln(2) = \ln(\sqrt{2})$.
In the simplest case $n=2$, $m$ even, we obtain
\[
\quad \left|  \ln(\sqrt{2})  - \left(  \frac{1}{2} - \frac{1}{4} + \frac{1}{6} - \frac{1}{8} \pm \ldots  \pm \frac{1}{m} \right)  \right| < \frac{1}{2(m+1)}
\]
and
\[
\quad \left|  \ln(\sqrt{2})  - \left(  \frac{1}{2} - \frac{1}{4} + \frac{1}{6} - \frac{1}{8} \pm \ldots  \pm \frac{1}{m-2} \right)  \right| > \frac{1}{m} - \frac{1}{2m+1},
\]
cf.\ proof of Proposition~\ref{P1}.
Using now the substitutions $m=2k$ and $m=2k+2$, respectively,
we get
\[
 \frac{2k+3}{8k^2+18k+10} < \left|  \ln(\sqrt{2}) - S_k  \right| < \frac{1}{4k+2},
\]
where $S_k$ denotes the $k$th partial sum
\[
S_k  := \sum_{i=1}^{k} (-1)^{i+1} \frac{1}{2i}.
\]
As in Section \ref{S2}, increasing $n$ improves the estimates,
e.g.\ $n=4$, $m$ odd, gives
\[
\frac{1}{4} \left( -\frac{1}{2k+6}+\frac{4}{2k+5}-\frac{5}{2k+4}+\frac{4}{2k+2} \right) < \left|  \ln(\sqrt{2}) - S_k  \right| < \frac{1}{4} \left( \frac{1}{2k+4} - \frac{4}{2k+3} + \frac{5}{2k+2} \right)
\]
or equivalently
\[
\frac{4k^3+32k^2+87k+83}{16k^4+136k^3+416k^2+536k+240} < \left|  \ln(\sqrt{2}) - S_k  \right| < \frac{4k^2+16k+17}{16k^3+72k^2+104k+48}.
\]
Multiplying the inequalities by 2, we now easily get error estimates for the alternating harmonic series
\[
\ln(2) = 1 - \frac{1}{2} + \frac{1}{3} - \frac{1}{4} \pm \ldots 
\]
Denoting by $AHS_k$ its $k$th partial sum, we conclude

\begin{Proposition} 
\[
 \frac{2k+3}{4k^2+9k+5} < \left|  \ln(2) - AHS_k  \right| < \frac{1}{2k+1},
\]
\end{Proposition}

\begin{Proposition} 
\[
\frac{1}{2} \left( -\frac{1}{2k+6}+\frac{4}{2k+5}-\frac{5}{2k+4}+\frac{4}{2k+2} \right) < \left|  \ln(2) - AHS_k  \right| < \frac{1}{2} \left( \frac{1}{2k+4} - \frac{4}{2k+3} + \frac{5}{2k+2} \right)
\]
or equivalently

\[
\frac{4k^3+32k^2+87k+83}{8k^4+68k^3+208k^2+268k+120} < \left|  \ln(2) - AHS_k  \right| < \frac{4k^2+16k+17}{8k^3+36k^2+52k+24}.
\]
\end{Proposition}

\end{document}